\documentclass[12pt,reqno]{amsart}

\usepackage[usenames]{color} 
\usepackage{ulem} 

\usepackage{pbruillard_commands}

\newcommand{\slf}{\mathfrak{sl}}
\newcommand{\mbZ}{\mathbb{Z}}

\newcommand{\B}{\mathcal{B}}

\newcommand{\Z}{\mathbb{Z}}

\newcommand{\one}{\mathbf{1}}
\newcommand{\ot}{\otimes}
\newcommand{\inv}{^{-1}}
\newcommand{\ol}{\overline}
\newcommand{\tS}{\tilde{S}}

\newcommand{\op}{\oplus}
\newcommand{\hS}{\hat{S}}

\calclayout

\title[Classification of super-modular categories by rank]{Classification of super-modular categories by rank}
\date{\today}
\author[P. Bruillard]{Paul Bruillard}
\email{Paul.Bruillard@pnnl.gov}
\address{Pacific Northwest National Laboratory, 902 Battelle Boulevard,
Richland, WA U.S.A}

\author[C. Galindo]{C\'{e}sar Galindo}
\email{cn.galindo1116@uniandes.edu.co}
\address{Departamento de Matem\'aticas, Universidad de los Andes, Bogot\'a, Colombia.}

\author[S.-H. Ng]{Siu-Hung Ng}
\email{rng@math.lsu.edu}
\address{Department of Mathematics, Louisiana State University, Baton Rouge, LA
    U.S.A.}
\author[J. Plavnik]{Julia Y. Plavnik}
\email{julia@math.tamu.edu}
\address{Department of Mathematics,
    Texas A\&M University,
    College Station, TX
    U.S.A.}
\author[E. Rowell]{Eric C. Rowell}
\email{rowell@math.tamu.edu}
\address{Department of Mathematics,
    Texas A\&M University,
    College Station, TX
    U.S.A.}
\author[Z. Wang]{Zhenghan Wang}
\email{zhenghwa@microsoft.com}
\address{Microsoft Research Station Q and Department of Mathematics,
    University of California,
    Santa Barbara, CA
    U.S.A.}
\begin{document}
\begin{abstract}
We pursue a classification of low-rank super-modular categories parallel to that of modular categories.
We classify all super-modular categories up to rank=$6$, and spin modular categories up to rank=$11$.  In particular, we show that, up to fusion rules, there is exactly one non-split super-modular category of rank $2,4$ and $6$, namely $PSU(2)_{4k+2}$ for $k=0,1$ and $2$. This classification is facilitated by adapting and extending well-known constraints from modular categories to super-modular categories, such as  Verlinde and Frobenius-Schur indicator formulae.
\end{abstract}
\thanks{The results in this paper were mostly obtained while all six authors were at the American Institute of Mathematics, participating in a SQuaRE.  We would like to thank that institution for their hospitality and encouragement. Galindo was partially supported by the Ciencias B\'asicas funds from vicerrectoria de investigaciones de la Universidad de los Andes, Ng by NSF grant DMS-1501179, Plavnik by CONICET, ANPCyT and Secyt-UNC, Rowell and Plavnik by NSF grant DMS-1410144, and Wang by NSF grant DMS-1411212. The research described in this paper was, in part, conducted under the Laboratory Directed Research and Development Program at PNNL, a multi-program national laboratory operated by Battelle for the U.S. Department of Energy. \textit{PNNL Information Release:} PNNL-SA-126378.
}
\maketitle

\section{Introduction}
\normalem

Elementary particles such as  electrons and photons are either fermions or bosons. But elementary excitations of topological phases of matter behave like exotic particles called anyons.  When the underlying particles of a topological phase of matter are bosons, the emergent anyon system is well modelled by a unitary modular category \cite{RSW}.  But most real topological phases of matter such as the fractional quantum Hall liquids are materials of electrons.  While a substantial part of the theory of anyons can be developed using unitary modular categories by bosonization,  to fully capture topological properties of anyons in fermion systems require super-modular categories \cite{16fold}.

Super-modular categories are unitary premodular categories with M\"uger center equivalent  to the unitary symmetric fusion category $\sVec$ of super-vector spaces. Both mathematically and physically, it is interesting to pursue a theory of super-modular categories parallel to modular categories and study problems such as rank-finiteness and classification.  Moreover, the general structure of unitary premodular categories is reduced to that of modular or super-modular categories via de-equivariantization \cite{BNRW,16fold}, which provides another motivation to study super-modular categories.

Unitary modular categories sit inside (split) super-modular categories as $\mathcal{C}\subset \mathcal{C}\boxtimes \sVec$.  However, the  degeneracy of the $S$-matrix for super-modular categories complicates their classification: many standard results for modular categories either fail or require significant modification.  Consequently fundamental problems such as rank-finiteness for super-modular categories are still open.

In this paper we pursue a classification of low-rank super-modular categories parallel to \cite{RSW}.   We classify all super-modular categories up to rank=$6$, and spin modular categories up to rank=$11$.  In particular, we show that, up to fusion rules, there is exactly one non-split super-modular category of rank $2,4$ and $6$, namely $PSU(2)_{4k+2}$ for $k=0,1$ and $2$.  We also show that rank-finiteness for unitary premodular categories would be a consequence of the minimal modular extension conjecture for super-modular categories \cite[Conjecture 3.14]{16fold}.

\section{Super-modular Categories and Fermionic Quotients}

Recall \cite{BK} that a \textbf{premodular} category $\mcB$ is a braided fusion category with a chosen spherical pivotal structure.  The isomorphism classes of simple objects will be labeled by $\Pi:=\Pi_\mcB=\{0,\ldots, r\}$ where $\1\cong X_0$ is the monoidal unit object and $X_i$ will be a chosen representative of the class $i$.  The unnormalized $S$-matrix will be denoted $\tS$ to distinguish it from the normalized version: $S=\frac{\tS}{D}$ where $D^2=\dim(\B)$ with $D>0$.  Notice that the categorical dimensions of the simple objects $X_i$ are $d_i:=\tS_{0,i}$, which are strictly positive for unitary categories, and $\sum_{i\in\Pi}d_i^2=D^2$.  The twists of the simple objects are $\theta_i$.

The \textbf{M\"uger centralizer} of a subcategory $\mcD\subset\mcB$ of a premodular  category $\mcB$ is the subcategory $C_\mcB(\mcD)$ generated by the simple objects $W\in\mcB$ such that ${\tilde{S}_{W,X}=d_Wd_X}$ for all $X\in\mcD$ (see \cite[Corollary 2.14]{M2}), and the \textbf{M\"uger center} of $\mcB$ is $C_\mcB(\mcB)=\mcB^\prime$.  A \textbf{modular} category $\mcC$ has trivial M\"uger center, \textit{i.e.}  $\mcC^\prime\cong\Vec$ whereas a \textbf{symmetric category} $\mcS$ has $\mcS^\prime=\mcS$.   Clearly $\mcB^\prime$ is itself a symmetric fusion category for any premodular category. The category of super-vector spaces is the fusion category of $\mbZ_2$-graded finite-dimensional vector spaces equipped with the braiding 
given by $c_{V_i, V_j} = (-1)^{ij} \tau$, for any homogeneous vector spaces $V_i, V_j$ of degree $i$ and $j$ respectively, where $\tau$ is the usual flip map of vector spaces.
 The symmetric fusion category of super-vector spaces has a unique spherical structure so that the dimensions are strictly positive, and we denote this premodular category $\sVec$ and its unnormalized $S$-matrix is 
 $\tS_{\sVec}=\begin{pmatrix} 1 & 1\\ 1& 1\end{pmatrix}$.  A non-trivial simple object in $\sVec$ is called a \textbf{fermion} and we typically denote a representative by $f$.  It is easy to see that we must have $\theta_f=-1$ for this unitary spherical structure.    
A unitary premodular category $\mcB$ with $\mcB^\prime\cong\sVec$ is called \textbf{super-modular} \cite{16fold}.
A super modular category $\mcB$ is called \textbf{split} if there is a modular category $\mcC$ so that $\mcB\cong \mcC\boxtimes \sVec$, and otherwise it is \textbf{non-split}.  For example, $\sVec$ itself is split since we may take $\mcC\cong\Vec$ the trivial modular category of vector spaces.

\begin{remark} \begin{enumerate}
\item More generally one defines the M\"uger center $\mcB^\prime$ of a braided tensor category $\mcB$ as the full subcategory generated by simple objects $X$ such that $c_{Y,X}c_{X,Y}=Id_{X\ot Y}$ for all objects $Y$.  Then a braided fusion category $\mcB$ with $\mcB^\prime$ equivalent to the symmetric category of super-vector spaces is called \textit{slightly degenerate} in \cite{ENO2}.  As we restrict our attention to unitary premodular categories this is equivalent to super-modular.  
\item Let $\sVec^{-}$ denote the other (non-unitary) spherical symmetric fusion category obtained from the category of super-vector spaces.  We do not know of any  premodular categories $\mcB$ with $\mcB^\prime\cong\sVec^{-}$ that does not split as $\mcC\boxtimes\sVec^{-}$.  On the other hand, it is easy to construct non-unitary premodular categories $\mcB$ with $\mcB^\prime\cong\sVec$ that do not split (via Galois conjugation, for example).  
\item As we do not use the Hermitian structure on our categories, all of our results hold under the (possibly weaker) assumption that the objects have positive dimensions.  We will always assume the dimensions are postive unless otherwise noted.
\end{enumerate}
\end{remark}

\begin{definition} A \textbf{fusion rule of rank $r+1$} is the collection of matrices $\mcN:=\{N_i:0\leq i\leq r\}$ so that $(N_i)_{k,j}=N_{i,j}^k$ correspond to a unital based ring (with unit $0$) in the sense of {\cite[Definition 2.2]{O1}}.  In particular there is an involution ${}^*$ on the labels $0\leq i\leq r$ so that $N_{i^*}=(N_i)^T$.
A fusion rule is {\bf commutative} if $N_iN_j=N_jN_i$ for any $i,j$, in which case each $N_i$ is a normal matrix.  A \textbf{ mock $S$-matrix} $S=(S_{ij})$ of a commutative fusion rule is a unitary simultaneous diagonalizer of $\mcN$.
\end{definition}
Two fusion categories $\mcD$ and $\mcC$ are called \textbf{Gothendieck equivalent} if their fusion rules are isomorphic, i.e. if there is an isomorphism of Grothendieck semirings $K_0(\mcD)\cong K_0(\mcC)$.

\subsection{Properties}

Fusion rules have been studied in other contexts such as table algebras \cite{Blau} and
association schemes \cite{Bannai}.  Indeed, the following result can be proved by a careful application of results in \cite[Theorem 4.1]{Bannai}: 

\begin{theorem} \label{t:1}
Let $\mcN$ be any commutative fusion rule.
\begin{enumerate}
    \item {Let $S$ be a simultaneous diagonalizer of $\mcN$. Then a complex square matrix $S'$ is a simultaneous diagonalizer of $\mcN$ if and only if  $S'=SD'P$ for some permutation matrix $P$ and a nonsingular diagonal matrix $D'$.}

\item {If $S$ is a  symmetric mock $S$-matrix of $\mcN$, then it satisfies the Verlinde rule}:
$$N_{ab}^c=\sum_{j}\frac{S_{aj}S_{bj}\bar{S}_{cj}}{S_{0j}}.$$\end{enumerate}
\end{theorem}

\begin{proof} 
{(i) Let $\mcU$ be the unital based ring defined by $\mcN$ with basis $\{x_0, \dots, x_r\}$.
  Suppose $S$ is a diagonalizer of $\mcN$.  Then $S\inv N_i S = D^{(i)}$ for all $i$, where $D^{(i)}$ is a diagonal matrix.
  Then the map $\phi_j(x_i) = D^{(i)}_{jj}$, $i=0,\dots r$, defines a character of $\mcU$ for each $j$, and $\{\phi_j \mid j =0, \dots, r\}$ is the set of all irreducible characters of $\mcU$. If $S'$ is a simultaneous diagonalizer of $\mcN$, then there exists a permutation matrix $P$ such that $P {S'}\inv N_i S'P\inv =D^{(i)}=S\inv N_i S$ for all $i$.  Therefore,
  \begin{equation} \label{eq:commutativity}
     S\inv S'P\inv D^{(i)} = S\inv N_i S'P\inv =  D^{(i)} S\inv S'P\inv\,.
  \end{equation}
  Suppose $k, l$ are distinct elements of $\{0, \dots, r\}$. There exist $i$ such that $\phi_k(x_i) \ne \phi_j(x_i)$. Hence, $D^{(i)}_{jj} \ne D^{(i)}_{kk}$. The equation \eqref{eq:commutativity} implies $S\inv S'P\inv=D'$ is an invertible diagonal matrix. The converse of the statement is an immediate and direct verification.
  
  (ii) Let $S$ be a symmetric mock $S$-matrix of $\mcN$ and $S'$ be a simultaneous diagonalizer of $\mcN$ given in \cite[Theorem 4.1]{Bannai}. Then there exist a permutation $\s$ on $\{0,\dots, r\}$ and a nonsingular diagonal matrix $D'$ such that 
  \begin{equation} \label{eq:diagonal}
  S'Q = S D' 
  \end{equation}
  where $Q$ is the permutation matrix associated with $\s$. By \cite[Theorem 4.1]{Bannai},
  $$
  S^{\prime\inv} N_i S^\prime = \lambda^{(i)}
  $$
  for all $i$, where $\lambda^{(i)}_{jk}=\frac{S'_{ij}}{S'_{0j}}\delta_{jk}$. Suppose $S\inv N_i S = \b^{(i)}$ for each $i$, where $\b^{(i)}$ is a diagonal matrix.  Then,
  $$
  \b^{(i)} = (D')\inv S\inv N_i SD'= Q\inv {S'}\inv N_i S' Q = Q\inv \lambda^{(i)} Q = \left(\frac{S'_{i\s(j)}}{S'_{0\s(j)}}\delta_{jk}\right)_{jk}=
  \left(\frac{S_{ij}}{S_{0j}}\delta_{jk}\right)_{jk}
  $$
  The last equality follows from \eqref{eq:diagonal}.
 Thus, we find
 $$
 N_i  = S \b^{(i)} S\inv \,.
 $$
 Since $S$ is a symmetric unitary matrix, the equation implies the Verlinde rule is satisfied by $S$. 
  }
  \end{proof}

In particular for modular categories the usual (normalized) $S$-matrix is a mock $S$-matrix.  This implies:
\begin{corollary}
Suppose two modular categories $\mcC$ and $\mcD$ are Grothendieck equivalent.  Then their $S$-matrices are equal via a permutation and rescaling of columns {in the sense of Theorem \ref{t:1}(i).}
\end{corollary}

\subsection{Super Modular Categories and Fermionic Modular Quotients}\label{Mock S-matrix}

Let $\B$ be a super-modular category, with fermion $f$ {and set of isomorphism classes of simple objects $\Pi$}.  Since {the action of $f$ on $\Pi$} is fixed-point-free, we may partition $\Pi$ into two sets $\Pi_0\bigsqcup f\Pi_0$.  We may arrange this (non-canonical) partition so that $0\in\Pi_0$ where $X_0=\one$ and if $X\in\Pi_0$ then so is $X^*$.  Indeed, if $X\in\Pi_0$, $f\ot X\not\cong X^*$ since their twists have opposite signs. We label the simple classes as follows: $0,1,\ldots,r\in\Pi_0$, $f,f\cdot1,\ldots,f\cdot r\in f\Pi_0$ where $X_{f\cdot i}:=f\ot X_i$ and $f\cdot 0=f$.  We will often denote $f\ot X$ by juxtaposition $f X$.
For each object $X_i$ the maximal eigenvalue of the corresponding fusion matrix $N_i$ is denoted $\FPdim(X_i)$ and $\FPdim(\B):=\sum_{i\in\Pi} \FPdim(X_i)^2$.  When $\FPdim(\B)=\dim(\B)$ we say that $\B$ is \textbf{pseudo-unitary}.
\begin{lemma}\label{lemma: factorization}
Let $\mcB$ be a (not necessarily pseudo-unitary) premodular category and $\mathcal{D}\subset \mcB$ a modular subcategory. If $\operatorname{FPdim}(\mathcal{D})=\frac{\operatorname{FPdim}(\mathcal{B})}{\operatorname{FPdim}(\mathcal{B}^\prime)}$ then $\mcB\cong \mathcal{D}\boxtimes \mcB^\prime$ as premodular categories. If $\mcB$ is unitary, by restriction $\mathcal{D}$ and $\mathcal{B}^\prime$ are also unitary and the equivalence  $\mcB\cong \mathcal{D}\boxtimes \mcB^\prime$ is of unitary premodular categories. 
\end{lemma}
\begin{proof}
By \cite[Theorem 4.1]{M2} we see that $\mcD\boxtimes C_\mcB(\mcD)$ is equivalent as a  premodular category to a fusion subcategory of $\B$. Since {$\operatorname{FPdim}(C_\mcB(\mcD))\geq \operatorname{FPdim}(\mcB^\prime)$} and $\operatorname{FPdim}(\mathcal{D})=\frac{\operatorname{FPdim}(\mathcal{B})}{\operatorname{FPdim}(\mathcal{B}^\prime)}$, we have that $\operatorname{FPdim}(\mathcal{B})=\operatorname{FPdim}(\mathcal{D}\boxtimes \mcB^\prime)$. Hence, $C_\mcB(\mcD)\cong \mcB^\prime$ and $\mcB$ is equivalent to $\mathcal{D}\boxtimes \mcB^\prime$. Since unitarity passes to subcategories, the proof is complete.
\end{proof}

\begin{theorem}\label{graded_splits}
Let $\mcB=\mcB_0\oplus \mcB_1$ be a $\mathbb{Z}_2$-graded super-modular category. If $f\in \mcB_1$ then  $\mcB\cong\mcB_0\boxtimes\sVec$ as unitary premodular categories with $\mcB_0$ modular.
\end{theorem}
\begin{proof}
Let ${Y}\in \mcB_0$ and $W\in C_{\mcB_0}(\mcB_0)$ be simple objects. Using \cite[Lemma 2.4]{M2}, we have 
\begin{align*}
S_{f\otimes Y,W}&=\frac{1}{d_W}S_{f,W}S_{Y,W}\\
&=\frac{d_fd_Wd_Yd_W}{d_W}\\
&=d_{f\otimes Y}d_{W}.
\end{align*} Hence $C_{\mcB_0}(\mcB_0)\subset \mcB^\prime$. Since $f\notin \mcB_0$, $C_{\mcB_0}(\mcB_0)=\operatorname{Vec}$, that is,  $\mcB_0$ is modular. Applying Lemma \ref{lemma: factorization} we yield the result.
\end{proof}


For each $i,j,k\in\Pi_0$ we define $$\hat{N}_{i,j}^k=\dim\Hom(X_i\ot X_j,X_k)+\dim\Hom(X_i\ot X_j,f\ot X_k)=N_{i,j}^k+N_{i,j}^{f\cdot k}.$$  These $\hat{N}_{i,j}^k$ are called the \textbf{naive fusion rules} for the fermionic quotient, and  they define {a unital based ring $\hat{\mcU}_\B$ of rank $r+1$ by  Proposition \ref{mock S-matrix properties}.  By \cite[Theorem 3.9]{16fold}, the unnormalized $S$-matrix of $\B$ has the form: $\tS=\begin{pmatrix} \hat{S} & \hat{S}\\ \hat{S} & \hat{S}
\end{pmatrix}$ where $\hS$ is symmetric and invertible.  We will see that $\hS$ is projectively unitary: $\hS\ol{\hS}=\frac{D^2}{2}I$ and will be called the \textbf{$S$-matrix of the fermionic quotient}.  This designation is justified by the following:

\begin{prop}\label{mock S-matrix properties} Let $\Pi_0, \hat{S}$ and $\hat{\mcN}=\{\hat{N}_i:0\leq i\leq r\}$ be as above for a given {super-modular} category $\B$.  Then:
\begin{enumerate}
    \item[(a)] $\hat{S}$ is symmetric and ${\hat S}\inv= \frac{2}{D^2} \ol{{\hat S}}$.
    \item[(b)] {$\hat{\mcN}$ is a commutative fusion rule.}
    \item[(c)] {Let $\{x_i \mid i \in \Pi_0\}$ denote the basis of $\hat{\mcU}_\B$.} Then the functions $\varphi_i(x_j):=\hat{S}_{ij}/\hat{S}_{0i}$ for $0\leq i\leq r$ form a set of orthogonal characters of the fusion algebra $\hat{\mathcal{U}}_\B$, and so $\hat{S}$ simultaneously diagonalizes the matrices $\hat{N}_i$.
    \item[(d)] $\displaystyle \hat{N}_{ij}^k=\sum_{{m \in \Pi_0}}\frac{\frac{2}{D^2}\hat{S}_{im}\hat{S}_{jm}\overline{\hat{S}_{km}}}{\hat{S}_{0m}}$.
    
\end{enumerate}

\end{prop}
Note that one interpretation of (c) is that {$\frac{\sqrt{2}}{D}\hat{S}$} is a mock $S$-matrix for the fusion algebra $\hat{\mathcal{U}}_\B$ associated with the fusion rule $\hat{\mathcal{N}}$ of the fermionic quotient of $\B$.
\begin{proof}
It is immediate that $\hat{S}$ is symmetric since  $\tS$ is symmetric.  According to \cite[Lemma 2.15]{M2}, for simple objects  $Y,Z$ we have 
$${\sum_{X \in\Pi} S_{X,Y}S_{X,Z}=N_{Y,Z}^\one + N_{Y,Z}^f} $$
since the simple objects in $\B'$ are $\one$ and $f$. By choosing $j,k\in\Pi_0$ and setting $X=X_j,Y=X_k$ we then have: $$\sum_{i\in\Pi}S_{i,j}S_{i,k}=\frac{2}{D^2}\sum_{i\in\Pi_0}\hat{S}_{j,i}\hat{S}_{i,k}=
N_{j,k}^0+N_{j,k}^f=\hat{N}_{j,k}^0.$$
Now since we have chosen $\Pi_0$ to be closed under ${}^*$ we see that $\hat{N}_{j,k}^0=\delta_{j,k^*}$ and so  $\hat{S}^2=\frac{D^2}{2} C$ where $C$ is the  charge conjugation matrix: $C_{i,j}=\delta_{i,j^*}$.  Thus we have proved (a). Notice that we have $\ol S_{ij} = S_{ij^*}$ for premodular categories since we may embed them in their (modular) Drinfeld center.  Statement (b) can be verified directly, but is also a consequence of (c).  

Statement (c) is a consequence of the fact that the (normalized) columns of the $S$-matrix {$\tilde S$} of any {premodular} category are characters {of its Grothendieck ring}.  Indeed, fixing $i,j,k\in\Pi_0$ we have (\cite[Lemma 2.4(iii)]{M2}):

$${\frac{\tilde S_{i,j} \tilde S_{i,k}}{\tilde S^2_{0,i}}=\sum_{m\in\Pi}N_{j,k}^m\frac{\tilde S_{i,m}}{\tilde S_{0,i}}.}$$
Splitting the right-hand-side into $m\in\Pi_0$ and $fm\in f\Pi_0$ and observing that {$\tS_{i,m}=\tS_{i,f\cdot m}$} we obtain: 
{
\begin{equation}\label{orthog}
   \frac{\hat{S}_{i,j}\hat{S}_{i,k}}{\hat{S}^2_{0,i}}=\sum_{m\in\Pi_0}\hat{N}_{j,k}^m\frac{\hat{S}_{i,m}}{\hat{S}_{0,i}}.
\end{equation}

The equation means that $\{\varphi_i\mid i \in\Pi_0\}$ is a set of irreducible characters of $\hat{\mcU}_\B$. In fact, this is the set of \emph{all} irreducible characters of $\hat{\mcU}_\B$ since {$(\hat{S}_{ij}/\hat{S}_{0i})_{ij}$} is also invertible. Equation \eqref{orthog} also implies that the column vector $(\hat{S}_{i,m})_{m\in\Pi_0}$ is an eigenvector for $\hat{N}_j$ with eigenvalue ${\frac{\hat{S}_{i,j}}{\hat{S}_{0,i}}}$ for all $i, j\in \Pi_0$. Thus, we have the matrix equation  
$$
\hat{N}_i \hat{S} = \hat{S} \lambda^{(i)} 
$$
for all $i \in \Pi_0$, where $\lambda^{(i)}_{jk} = \frac{\hat{S}_{i,j}}{\hat{S}_{0,i}} \delta_{jk}$.
Now (d) follows from this equation and (a).
}
\end{proof}

\begin{lemma}\label{FS-indicator super}
    For $X_{j}$ a simple self-dual object in a {super-modular} category $\B$, we have
    \begin{align*}
      \pm 1&=\n_{2}(X_{j})=\frac{2}{{\dim \mcB}}\Sum_{a,b {\in \Pi_0}}\hat{N}_{a,b}^{j}d_{a}d_{b}\left(\frac{\th_{a}}{\th_{b}}\right)^{2},
    \end{align*}
     where $\hat{N}_{i,j}^k$ are the naive fusion rules for the fermionic quotient.
  \end{lemma}
  \begin{proof} {The first equality follows directly from \cite{NS}.}
  Recall from \cite{B2} that the formula of the second {Frobenius}-Schur indicator of a self-dual simple object in a premodular category is the following
  \begin{equation*}
      \n_{2}(X_{j})
      =\frac{1}{{\dim \mcB}}\Sum_{a,b {\in \Pi}}N_{a,b}^{j}d_{a}d_{b}\left(\frac{\th_{a}}{\th_{b}}\right)^{2} - \th_{j} \Sum_{k\in {\Irr(\B')}\setminus \{\textbf{1}\}} d_k \Tr (R_k^{j j}).
      \end{equation*}
      
      The unique non-trivial transparent object in $\B$ is the fermion $f$ since $\B$ is {super-modular}, so that
     { 
    \begin{align*}
      \n_{2}(X_{j})
      &=\frac{1}{\dim \mcB}\Sum_{a,b \in \Pi} N_{a,b}^{j}d_{a}d_{b}\left(\frac{\th_{a}}{\th_{b}}\right)^{2} - \th_{j} \Tr (R_f^{j j}) \\
      \end{align*}
       Since a transparent fermion does not fix simple objects and $X_j$ is self-dual, $f$ is not a subobject of $X_j\otimes X_j$. In particular the term $\th_{j} \Tr (R_f^{j j})$ vanishes and so
      \begin{align*}
      \n_{2}(X_{j})
      &=\frac{1}{\dim \mcB}\Sum_{a,b \in \Pi_0}\left(N_{a,b}^{j}+N_{a,f\cdot b}^{j}+N_{f \cdot a,b}^{j}+N_{f \cdot a,f\cdot b}^{j}
      \right) d_{a}d_{b}\left(\frac{\th_{a}}{\th_{b}}\right)^{2}\\
      &=\frac{1}{\dim \mcB}\Sum_{a,b \in \Pi_0}\left(N_{a,b}^{j}+N_{a,b}^{f\cdot j}+N_{a,b}^{f\cdot j}+N_{a,b}^{j}\right)d_{a}d_{b}\left(\frac{\th_{a}}{\th_{b}}\right)^{2}.\\
    \end{align*}
    }
   Since the naive fusion rules are given by $\hat{N}_{i,j}^k=N_{i,j}^k+N_{i,j}^{f\cdot k}$, the statement of the lemma now follows. 
  \end{proof}
  
  \section{Low Rank Super-Modular Categories}
 The first examples of non-split super-modular categories are the rank $2(k+1)$ adjoint subcategories $PSU(2)_{4k+2}\subset SU(2)_{4k+2}$  (see \cite{16fold}).  The modular categories $SU(2)_{4k+2}$ are  constructed as subquotient categories of representations of quantum groups $U_q\mathfrak{sl}_2$ with $q=e^{\frac{\pi i}{4k+4}}$.  Replacing $q$ by $q^t$ with $(t,4k+4)=1$ yields new categories with the same fusion rules, which may or may not be unitary.  In the notation of \cite{RowellUMC}, this family of modular categories would be denoted $\mcC(\mathfrak{sl}_2,q^t,4k+4)$.  We will use similar notation for the adjoint subcategories: $\mcC(\mathfrak{psl}_2,q^t,4k+4)$.  The two we will encounter here are for $k=1,2$.
 
 For $PSU(2)_6$ we label the simple objects $\one,X_1,fX_1,f$: this ordering conforms with the natural ordering of objects by highest weights in $\mathfrak{su}_2$. 
 The fusion rules can be derived from:
 $$f^{\otimes 2}\cong\one,\quad X_1^{\otimes 2}\cong \one\oplus X_1\oplus fX_1$$ and one sees that $d_1=1+\sqrt{2}$.   
 
 For $PSU(2)_{10}$ we have simple objects $\one,X_1,X_2,fX_2,fX_1,f$ (again, this ordering conforms with the standard ordering by highest weights in $\mathfrak{su}_2$)  with all fusion rules consequences of the following:
 \begin{itemize}
 \item $f^{\otimes 2}\cong \one$,
     \item $X_1^{\otimes 2}\cong\one\oplus X_1\oplus X_2$,
     \item $X_1\otimes X_2\cong X_1\oplus X_2\oplus fX_2$,
     \item $X_2^{\otimes 2}\cong\one\oplus X_1\oplus X_2\oplus fX_1\oplus fX_2$.
 \end{itemize}
  From this one computes that $d_1=1+\sqrt{3}$ and $d_2=2+\sqrt{3}$.

The goal of this subsection is {to} classify all non-split super-modular categories of rank$\leq 6$, namely:
  \begin{theorem}
Any non-split super-modular category of rank $4$ or $6$ is Grothendieck equivalent to $PSU(2)_6$ or $PSU(2)_{10}$, respectively.  Moreover, any such category is of the form $\mcC(\mathfrak{psl}_2,q^t,8)$ or $\mcC(\mathfrak{psl}_2,q^t,12)$ {with $(t,4k+ 4) = 1$}.
\end{theorem}

\begin{proof}
By Lemma \ref{nonsplit rank 6} below, we have that any non-split super-modular category of rank $6$ is Grothendieck equivalent to $PSU(2)_{10}$ whereas the rank $4$ case is contained in  \cite[Proposition 4.10]{B2}.   Now by \cite[Corollary 8.8]{MPS} (see also \cite[Theorem 3.4]{GS}) any braided fusion category Grothendieck equivalent to $PSU(2)_\ell$ for $\ell\geq 5$ is of the form $\mcC(\mathfrak{psl}_2,q^t,\ell)$ (The notation used there is $SO(3)_q$).  This completes the proof.
\end{proof}

\begin{remark}
We caution the reader that it is \text{not} the case that any modular category Grothendieck equivalent to $SU(2)_k$ is of the form $\mcC(\slf_2,Q,2k+4)$ for some primitive $(k+2)$th root of unity $Q$.  Indeed, \cite{KW} show that one may twist the associativity morphisms to obtain categories not of this form.  However, these twists have no effect on the subcategories $\mcC(\mathfrak{psl}_2,q^t,\ell)$, which explains why the classification from \cite{MPS} only depends on the choice of a root of unity.
\end{remark}

Split super-modular categories can easily be classified in these ranks, using the classification of modular categories of rank$\leq 3$ \cite{RSW}.  We recall that, up to fusion rules, modular categories of ranks $2$ and $3$ are \cite[Section 5.3]{RSW}:
\begin{enumerate}
\item $SU(2)_1$ (Semion)
\item $PSU(2)_3$ (Grothendieck equivalent to Fibonacci)
\item $SU(3)_1$ ($\Z_3$ theories)
\item $SU(2)_2$ (Grothendieck equivalent to Ising)
\item $PSU(2)_5$
\end{enumerate}

\subsection{Rank=6 Super-modular Categories}

In this subsection we will classify super-modular categories of rank $6$, up to Grothendieck equivalence.
We will follow the notation used in Subsection \ref{Mock S-matrix}. Let $\B$ be a super-modular category of rank $6$, with transparent fermion $f$.  Recall our partition of the isomorphism classes of simple objects $\Pi=\Pi_0\bigsqcup f\Pi_0$.   We have $d_{a}=d_{f\cdot a}$, $\hat{N}_{a,b}^{c}=N_{a,b}^{c}+N_{a,b}^{f\cdot c}$, and $\th_{a}=-\th_{f\cdot a}$ for all $a, b, c\in {\Pi_0}$.
 Notice that interchanging the labels $k$ and $f\cdot k$ for $k\neq 0$ (simultaneously interchanging duals if necessary) in the partition does not affect the fusion rules or the mock $S$-matrix but changes the signs of the corresponding $T$-matrix entries of the fermionic quotient.


  \begin{remark}\label{naive fusion matrix rank 6}
    If $\B$ is {a} self-dual super-modular category of rank $6$, then by  fusion rule symmetries and self-duality, the (naive) fusion matrices of the fermionic quotient are
    \begin{align*}
     \hat{N_{1}}&=\left(\begin{smallmatrix}0&1&0\\1&m&k\\0&k&\ell\end{smallmatrix}\right)
      ,\quad
     \hat{N_{2}}=\left(\begin{smallmatrix}0&0&1\\0&k&\ell\\1&\ell&n\end{smallmatrix}\right)
    \end{align*}
    for some {non-negative} integers $k,\ell,m,n$.
  \end{remark}

  \begin{lemma}
    \label{Lemma: Self-dual quotient}
    If $\B$ is self-dual super-modular category of rank $6$, then its fermionic quotient satisfies one of the following:
    \begin{itemize}
      \item[(i)] The dimensions are $d_{2}=1$ and $d_{1}^{2}=2$ and the naive fusion rules {$\hat{N}_j$} are the fusion rules
      of Ising;
      \item[(ii)] $d_{1}$ is a real root of $x^{3}-2x^{2}-x+1$ and
      $d_{2}=d_{1}/(d_{1}-1)$ is a root of $x^{3}-x^{2}-2x+1$, and the naive fusion rules are given by
      \begin{align*}
       \hat{N_{1}}&=
        \left(\begin{smallmatrix}
        0 & 1 & 0\\
        1 & 0 & 1\\
        0 & 1 & 1
        \end{smallmatrix}\right)
        ,\quad
        \hat{N_{2}} =
        \left(\begin{smallmatrix}
        0 & 0 & 1\\
        0 & 1 & 1\\
        1 & 1 & 1
        \end{smallmatrix}\right); 
      \end{align*}
      or
      \item[(iii)] there is $\a\in\mbbN$ such that the naive fusion rules (following the notation in Remark \ref{naive fusion matrix rank 6}) and the dimensions are given by
      \begin{align*}
        k&=2\a,\quad
        \ell=1,\quad
        m=2\a^{2},\quad
        n=\a\\
        d_{1}&=\a^{2}+1+\a\sqrt{\a^{2}+2}=1+\a d_{2}\\
        d_{2}&=\a+\sqrt{\a^{2}+2}.
      \end{align*}
    \end{itemize}
  \end{lemma}
  \begin{proof}
  The proof from \cite{RSW} follows \textit{mutandis mutatis} in this case.
  
  Some important facts that are needed for this proof but are still true in the fermionic modular case are that the mock $S$-matrix has orthogonal columns and is symmetric, by Proposition \ref{mock S-matrix properties}, and the Galois group of the field generated by $\hat{S}$ is abelian since the $S$-matrix entries are obtained from a premodular category.  
  \end{proof}

\begin{remark}
Notice that if $\a = 0$ in Lemma \ref{Lemma: Self-dual quotient}(iii), then the fermionic quotient has the fusion rules of Ising (it recovers case (i) in Lemma \ref{Lemma: Self-dual quotient}).
\end{remark}
\begin{lemma}
In Lemma \ref{Lemma: Self-dual quotient}(iii), $\a \leq 1$. 
\end{lemma}

\begin{proof}
Since $\B$ is self-dual, $\n_2(X_j) =\pm 1$, for $j= 0, 1, 2$. 

In case (iii) in Lemma \ref{Lemma: Self-dual quotient}, the naive fusion matrices are
 $\hat{N_{1}}=\left(\begin{smallmatrix}0&1&0\\1&2\a^2&2\a\\0&2\a&1\end{smallmatrix}\right)$, and 
 $\hat{N_{2}}=\left(\begin{smallmatrix}0&0&1\\0&2\a&1\\1&1&\a\end{smallmatrix}\right)$. It follows from the formula for the 2nd Frobenius-Schur indicator in Lemma \ref{FS-indicator super} that
$$\n_{2}(X_{2})
      =\frac{2}{D^{2}}\left(d_2 (\theta_2^2+\theta_2^{-2}) + 2\a d_1^2 + d_1d_2\left(\left(\frac{\theta_1}{\theta_2}\right)^2 +\left(\frac{\theta_1}{\theta_2}\right)^{-2} \right)+\a d_2^2\right). $$
{Assume $\a \ge 2$.
Since $d_{1} =1+\a d_{2} > d_2 > 2\a$ and $\n_{2}(X_{2})=\pm 1$, we have that
\begin{align*}
0  & = 2\a d_1^2 + \a d_2^2 + 2 d_1d_2 \Re\left(\frac{\theta_1}{\theta_2}
      \right)^2 + 2 d_2 \Re(\theta_2^2) \pm \frac{D^2}{2}  \\
  &\geq (2\a- 1) d_1^2 + (\a - 1) d_2^2 - 2 d_1 d_2 - 2 d_2 - 1\\
   &\geq 3 d_1^2 +  d_2^2 - 2 d_1 d_2 - 2 d_2 - 1\\
   &> d_1^2 +  (2\a-2) d_2   - 1 >0,
  \end{align*}
  which is a contradiction. Therefore, $\a\leq 1$ as stated.}
\end{proof}

\begin{corollary}\label{cor:naive}
If $\B$ is a self-dual super-modular category of rank 6, then it fermionic quotient satisfies one of the following:
\begin{enumerate}
    \item The $S$-matrix $\hat{S}$ of the fermionic quotient and naive fusion rules ${\hat{\mcN}}$ correspond to the Ising category.
    \item The $S$-matrix $\hat{S}$ of the fermionic quotient and naive fusion rules ${\hat{\mcN}}$ correspond to those of $PSU(2)_5$.
    \item The $S$-matrix of the fermionic quotient is of the form $\hat{S} = \left(\begin{smallmatrix}
          1 & 2+\sqrt{3} & 1+\sqrt{3}\\
          2+\sqrt{3} & 1 & -1-\sqrt{3}\\
          1+\sqrt{3} & -1-\sqrt{3} & 1+\sqrt{3}
        \end{smallmatrix}\right)$, the naive fusion rules are given by $\hat{N}_{1}= \left(\begin{smallmatrix}
          0 & 1 & 0\\
          1 & 2 & 2\\
          0 & 2 & 1
        \end{smallmatrix}\right)$ and $\hat{N}_{2}=\left(\begin{smallmatrix}
          0 & 0 & 1\\
          0 & 2 & 1\\
          1 & 1 & 1
        \end{smallmatrix}\right)$, and the dimensions are $d_1 = 2+\sqrt{3}$ and $d_2 = 1+\sqrt{3}$.
\end{enumerate}

\end{corollary}

\begin{lemma}\label{quotient Ising}
If $\B$ is a self-dual rank $6$ super-modular category whose fermionic quotient has {the same fusion rules as an} Ising category $\mcI$, then $\B\cong\mcI\boxtimes\sVec$.
\end{lemma}
\begin{proof}
By assumption the fermionic quotient of $\B$ {has the same fusion rules as an} Ising category $\mcI$ and it has simple objects $\one, \s,\ps$. Then $\B$ has simple objects $\one, \s,\ps, f, f\s, f\ps$.  

Since $1 = \hat{N}_{\s,\s}^{\ps} = N_{\s,\s}^{\ps} + N_{\s,\s}^{f\ps}$, the object $\s^2$ in $\B$ either contains $\ps$ or $f\ps$. Then, $\s^2 = \one \oplus \ps$ and $\ps\otimes \s = \s$ or $\s^2 = \one \oplus f\ps$ and $f\ps\otimes \s = \s$. Notice that it is always true that $(f\ps)^2 = \one = \ps^2$ in $\B$. Therefore, the subcategory of $\B$ generated by $\s$ is an Ising category, which is always modular by \cite[Corollary B.12]{DGNO1}. Then, it follows from \cite[Theorem 4.2]{M2} that $\B\cong \mcI\boxtimes \sVec$.
\end{proof}  

\begin{lemma}\label{quotient A}
If $\B$ is a self-dual rank $6$ super-modular category whose fermionic quotient has {the same fusion rules as}  $PSU(2)_5$, then $\B\cong\mcD\boxtimes\sVec$ where $\mcD$ is a $PSU(2)_5$ category. 
\end{lemma}
\begin{proof}
The simple objects of the quotient are denoted $\one, X_1$ and $X_2$, where $X_1$ is the $d$-dimensional object and $X_2$ has dimension $d^2-1$, where $d = 2\cos(\frac{\pi}{7})$ (see \cite[5.3.6]{RSW} for details on this category). We denote by $\one, X_1, X_2, f, fX_1, fX_2$ the simple objects in $\B$.

 Notice that when defining a fermionic quotient there is always a non-canonical labeling choice between objects $X$ and $fX$--indeed, we obviously have $\hat{N}_{X}=\hat{N}_{fX}$.  In particular when lifting naive fermionic fusion rules to a super-modular category we are free to interchange simple objects $X\leftrightarrow fX$ due to this labeling ambiguity.  For notational convenience we use $f\cdot i$ to label the matrix entry of $f\ot X_i$ in what follows. Since $1=\hat{N}_{2,2}^2=N_{2,2}^2+N_{2,2}^{f\cdot 2}$ and $N_{f\cdot 2,f\cdot 2}^{f\cdot 2}=N_{2,2}^{f\cdot 2}$ we may assume $N_{2,2}^2=1$ and $N_{2,2}^{f\cdot 2}=0$ by interchanging $X_2$ and $fX_2$ if necessary.  Similarly, we have $1=\hat{N}_{2,2}^1=N_{2,2}^1+N_{2,2}^{f\cdot1}$ so that interchanging $X_1$ and $fX_1$ allows us to assume $N_{2,2}^1=1$ and $N_{2,2}^{f\cdot 1}=0$.  After these two labeling choices, all remaining fusion rules for $\B$ can be derived from the following by tensoring with $f$:
        \begin{itemize}
        \item $f^{\ot 2}=\one$
        \item $X_1^{\ot 2}=\one\op aX_2\op b (fX_2)$ where $a+b=1$
        \item $X_1\ot X_2= aX_1\op X_2 \op b(fX_1)$ and 
        \item $X_2^{\ot 2}=\one\op X_1\op X_2$.
        \end{itemize}
        
Computing $X_1\ot X_2^{\ot 2}$ in two ways and comparing the multipicities of $X_1$ (which are $(1+a^2+b^2)$ and $(1+a)$) reveals that $2a^2-3a+1 =0$, which has solution $(a,b) = (1,0)$.  
Thus, the subcategory $\mcD$ of $\B$ generated by $X_1$ and $X_2$ has rank $3$ and is Grothendieck equivalent to $PSU(2)_5$. Moreover, $\B$ is graded with $\B_0 = \mcD$ and $\B_1 = f\mcD$. Therefore it follows Theorem \ref{graded_splits} that $\B\cong \mcD\boxtimes \sVec$, and $\mcD$ is a $PSU(2)_5$ modular category.
\end{proof}

\begin{lemma}\label{nonsplit rank 6}
If $\B$ is a self-dual rank $6$ category whose fermionic quotient is the one in Lemma Corollary \ref{cor:naive}(iii), then $\B$ is Grothendieck equivalent to $PSU(2)_{10}$.
\end{lemma}

\begin{proof}
After interchanging the labels $1$ and $2$ in Corollary \ref{cor:naive}(iii) we have the following naive fusion rules
$$\hat{N}_1=\left(\begin{smallmatrix}
          0 & 1 & 0\\
          1 & 1 & 1\\
          0 & 1 & 2
        \end{smallmatrix}\right)  \text{ and } \hat{N}_{2}=\left(\begin{smallmatrix}
          0 & 0 & 1\\
          0 & 1 & 2\\
          1 & 2 & 2
        \end{smallmatrix}\right).$$
        Since $1=\hat{N}_{1,1}^1=N_{1,1}^1+N_{1,1}^{f\cdot 1}$ and $N_{f\cdot 1,f\cdot 1}^{f\cdot 1}=N_{1,1}^{f\cdot 1}$ we may assume $N_{1,1}^1=1$ and $N_{1,1}^{f\cdot 1}=0$ by interchanging $X_1$ and $fX_1$ if necessary.  Similarly, we have $1=\hat{N}_{1,1}^2=N_{1,1}^2+N_{1,1}^{f\cdot2}$ so that interchanging $X_2$ and $fX_2$ allows us to assume $N_{1,1}^2=1$ and $N_{1,1}^{f\cdot 2}=0$.  After these two labeling choices, all remaining fusion rules for $\B$ can be derived from the following by tensoring with $f$:
        \begin{itemize}
        \item $f^{\ot 2}=\one$
            \item $X_1^{\ot 2}=\one\op X_1\op X_2$
            \item $X_1\ot X_2=X_1\op aX_2\op b(fX_2)$ where $a+b=2$ and 
            \item $X_2^{\ot 2}=\one\op aX_1\op cX_2\op d(fX_2)\op b(fX_1)$ where $c+d=2$.
        \end{itemize}
Computing $X_1^{\ot 2}\ot X_2$ in two ways and comparing the multipicities of $X_2$ (which are $(a+c+1)$ and $(1+a^2+b^2)$) reveals that $2\,{b}^{2}-3\,b+d=0$, which has solutions $(b,d)\in\{(0,0),(1,1)\}$.  But if $b=d=0$ then $a=c=2$ and $\one,X_1$ and $X_2$ form a premodular subcategory, contradicting \cite[Theorem 3.5]{O4}.  Thus $a=b=c=d=1$, which yeild the fusion rules of $PSU(2)_{10}$ described above.

\end{proof}

\begin{lemma}\label{non-self dual rank 6}
If $\B$ is a non-self dual super-modular category of rank $6$, then there is a primitive 3rd root of unity $\omega$ such that
$\tS=\hat{S}\otimes\left(\begin{smallmatrix}1&1\\1&1\end{smallmatrix}\right)$, and $\hat{S}=\left(\begin{smallmatrix}
      1 & 1 & 1\\
      1 & \w & \w^{2}\\
      1 & \w^{2} & \w
      \end{smallmatrix}\right).$
\end{lemma}
\begin{proof}
Since $\B$ is a non-self dual, rank $6$ super-modular category, its fermionic quotient has a non-self dual, rank $3$ fusion rule. 

The rest of the proof proceeds in the same way as the analysis in \cite[Appendix A.1]{RSW}. It is important to remark that we use that the {$S$-matrix of the fermionic quotient} $\hat{S}$ is symmetric and (projectively) unitary by Proposition \ref{mock S-matrix properties}. 
\end{proof}

\begin{corollary}
If $\B$ is a non-self dual rank $6$ super-modular category then $\B\cong\mcP_{3}\boxtimes\sVec$ where $\mcP_{3}$ is a cyclic rank $3$ pointed  modular category.
\end{corollary}
\begin{proof}
 Recall that at the beginning of this section, we arrange a (non-canonical) partition of the simple objects $\Pi$ into two sets $\Pi_0\bigsqcup f\Pi_0$ so that $\one\in\Pi_0$ and if $X\in\Pi_0$ then so is $X^*$. 
 Then, ${X_1^*} = X_2$. Since $X_2$ is non-self dual and  $X_2^* \neq f\otimes X_1$ then $X_2^{\otimes 2} \neq \one, f$. Therefore $X_2^{\otimes 2} = X_1$ or $fX_1$.  In the latter case we may simultaneously interchange $fX_2\leftrightarrow X_2$ and $fX_1\leftrightarrow X_1$ to reduce to $X_2^{\otimes 2} = X_1$.
 
 
 In this case the rank $3$ subcategory $\mcB_0$ generated by $X_2,X_1$ and $\one$ is modular by Lemma \ref{non-self dual rank 6}. Since its centralizer is $\sVec$ and {$f\in \B_1$} the result now follows from Theorem \ref{graded_splits}.
 
\end{proof}

\section{Low Rank Spin Modular Categories}
A \textbf{spin} modular category is a modular category $\mcC$ containing a fermion \cite{16fold}.
The key outstanding conjecture for super-modular categories is the following \cite[Conjecture 3.14]{16fold}:
\begin{conj}\label{GFPC}
Any super-modular category $\B$ is a ribbon subcategory of a (spin) modular category $\mcC$ with $\dim(\mcC)=2\dim(\B)$. 
\end{conj}

For a super-modular category $\B$, a modular category $\mcC$ satisfying these conditions is called a \textbf{minimal modular extension} of $\B$, and 
the results of \cite{KLW} imply that if one minimal modular extension exists for a given super-modular category $\B$ then there are precisely $16$ of them.  One consequence of this conjecture would be rank-finiteness for super-modular, and hence premodular categories.  To see this, we first prove a lemma, the idea of which came from Bonderson (see \cite{BCT,BRWZ}).

Let $\mcC$ be a spin modular category with fermion $f$.  The $\mbZ_2$ grading afforded by $\langle f\rangle\cong \sVec$ will be denoted $\mcC_0\oplus \mcC_1$, where $\mcC_0$ is super-modular with minimal modular extension $\mcC$.  {Since $f \in \mcC_0$,} a further refinement of $\mcC_1$ can be obtained by defining $\mcC_v\subset\mcC_1$ to be the abelian subcategory generated by simple objects $X$ with $f\ot X\not\cong X$ and $\mcC_\sigma\subset\mcC_1$ the abelian subcategory {generated by simple objects} $X$ with $f\ot X\cong X$. The following may be derived from \cite{BRWZ}, but we provide a proof for completeness:

\begin{lemma}\label{rankbounds} Let $(\mcC,f)$ be a spin modular category with $\mcC_0$, $\mcC_v$ and $\mcC_\sigma$ as above, {and their ranks denoted by $|\mcC_0|, |\mcC_v|, |\mcC_\sigma|$  respectively}. Then:
\begin{enumerate}
    \item[(a)] $|\mcC_0|=|\mcC_v|+2|\mcC_\sigma|$, in particular $|\mcC|=2|\mcC_0|-|\mcC_\sigma|$.
    \item[(b)] $3|\mcC_0|/2\leq |\mcC|\leq 2|\mcC_0|$.
    \item[(c)] $|\mcC_v|$ and $|\mcC_0|$ are even.
\end{enumerate}
\end{lemma}
\begin{proof}
We only prove $(a)$; the other two statements follow directly.

We partition the basis $\Pi=\Pi_0\bigsqcup f\Pi_0\bigsqcup \Pi_v\bigsqcup f\Pi_v\bigsqcup\Pi_\sigma$ for the Grothendieck ring of $\mcC$ so that $\mcC_0$ has basis $=\Pi_0\bigsqcup f\Pi_0$, $\mcC_v$ has basis $\Pi_v\bigsqcup f\Pi_v$ and $\mcC_\sigma$ has basis $\Pi_\sigma$.  With respect to this ordered basis the $S$ matrix of $\mcC$ has the block form:$$S=\begin{pmatrix}\frac{1}{2}\hat{S} & \frac{1}{2}\hat{S} & A&A &X\\
\frac{1}{2}\hat{S} & \frac{1}{2}\hat{S} & -A &-A &-X\\
A^T&-A^T & B& -B&0\\
A^T&-A^T &-B & B &0\\
X^T &-X^T &0 &0 &0
\end{pmatrix}$$
From this form we see that  $S$ maps the {subspace spanned by} the linearly independent set $\{X_i-f X_i: X_i\in\Pi_0\}$ bijectively to {the subspace spanned by} $\{Y_i+f Y_i: Y_i\in\Pi_v\}\cup \{Z_i\in\Pi_\sigma\}$.  The {dimension of the first subspace} is $|\mcC_0|/2$ whereas the latter two have  {dimensions} $|\mcC_v|/2$ and $|\mcC_\sigma|$ respectively, proving $(a)$.
\end{proof}

Now Conjecture \ref{GFPC} implies:

\begin{conj}
There are finitely many premodular categories of rank $r$.
\end{conj}
Indeed, it is enough to show there are finitely many super-modular categories by \cite{BNRW}.  If Conjecture \ref{GFPC} holds then every super-modular category of rank $r$ is a subcategory of a modular category of rank at most $2r$, of which there are finitely many by the main result of \cite{BNRW}.

\subsection{Classification of Spin Modular Categories of Rank$\leq 11$}

From the classification of super-modular categories of rank$\leq 6$ we obtain a classification of spin modular categories of rank at most $11$.  First we determine the ranks of components $\mcC_0$, $\mcC_v$ and $\mcC_\sigma$:

\begin{lemma}
Let $(\mcC,f)$ be a spin modular category and $\mcC_0$ the corresponding super-modular category.
\begin{enumerate}
    \item[(a)] If $|\mcC|=3,4$ or $5$ then $|\mcC_0|=2$.  Moreover, $\mcC\cong SO(N)_1$ in this case (and in particular $|\mcC|\neq 5$).
    \item[(b)] If $|\mcC|=6,7$ or $8$, then $|\mcC_0|=4$ and $(|\mcC_v|,|\mcC_\sigma|)=(0,2),(2,1)$ and $(4,0)$ respectively.
    \item[(c)] If $|\mcC|=9,10$ or $11$, then $|\mcC_0|=6$, and $(|\mcC_v|,|\mcC_\sigma|)=(0,3),(2,2)$ and $(4,1)$ respectively.
\end{enumerate}
\end{lemma}
\begin{proof}
The upper and lower bounds from Lemma \ref{rankbounds}(b) provide the value of $|\mcC_0|$ immediately.  Kitaev's \cite{kitaev} classification of spin modular categories of dimension $4$ finishes (a).  For the remaining statements use Lemma \ref{rankbounds}(a).
\end{proof}

In \cite[Section III.G]{16fold} all $16$ minimal modular extesnions of $PSU(2)_{4k+2}$ are explicitly constructed. Combining with the results above, we have:

\begin{theorem}\label{Spin Class}
Let $\mcC$ be a spin modular category of rank $|\mcC|\leq 11$.  Then either,
\begin{enumerate}
\item[(a)] {$\mcC\cong \mcD \boxtimes SO(N)_1$ for some postive integer $N\leq 16$ and $\mcD$ a modular category with $|\mcD| \le 3$} or,
\item [(b)] $|\mcC|=7$, and $\mcC$ is Grothendieck equivalent to one of the $16$ minimal modular extensions of $PSU(2)_6$ described in \cite[Section III.G]{16fold} or
\item[(b)] $|\mcC|=10$ or $11$ and $\mcC$ is Grothendieck equivalent to one of the $16$ minimal modular extensions of $PSU(2)_{10}$ described in \cite[Section III.G]{16fold}.
\end{enumerate}
\end{theorem}
\begin{proof}
If $K_0(\mcC_0)\not\cong K_0(PSU(2)_{4k+2})$ for $k=1,2$ then we have seen that $\mcC_0$ is split super-modular, i.e. of the form $\mcC_0\cong\sVec\boxtimes\mcD$ for some modular category $\mcD$.  Since there are exactly $16$ (or $0$) minimal modular extensions of any super-modular category and the extensions $\mcD\boxtimes SO(N)_1$ for $1\leq N\leq 16$ are minimal and distinct, we have proved (a).

If $\mcC_0$ is Galois conjugate to $PSU(2)_{4k+2}$ for $k=1,2$ then we extend the Galois automorphism $\sigma$ to $\mcC$  which changes $\mcC_0$ to $PSU(2)_{4k+2}$, apply the classification of \cite[Section III.G]{16fold}.
\end{proof}

Notice that if $|\mcC|=12$ then $|\mcC_0|=6$ or $8$, and in these cases we have  $(|\mcC_v|,|\mcC_\sigma|)=(6,0)$ or $(0,4)$ respectively.  Of course we may construct many such spin modular categories as Deligne products of $SO(N)_1$ with modular categories of rank $4$ (for $N$ odd) or rank $3$ (for $N$ even), but presumably there are others.  In fact, for $|\mcC_0|=6$ and $|\mcC|=12$ our classification implies that $\mcC_0$ must be split super-modular, hence $\mcC\cong\mcD\boxtimes SO(N)_1$, as a minimal modular extension of a split super-modular category.  
From the evidence we have seen so far (i.e. up to rank $6$) the following may be true:
\begin{conjecture}
    If $\mcB$ is super-modular with $S$-matrix
    $\hat{S}\otimes\begin{pmatrix}1&1\\1&1\end{pmatrix}$ with
    $\hat{S}$ the $S$-matrix of some modular category, $\mcD$, then
    $\mcB$ is split super-modular.
  \end{conjecture}

\bibliographystyle{plain}

\end{document}